\documentclass[a4paper,12pt]{article}

\usepackage{amsthm}
\usepackage{comment}
\usepackage{amsmath}
\usepackage{amssymb}
\usepackage{amsfonts}
\usepackage{mathrsfs} 
\usepackage{graphicx}
\usepackage[dvipsnames]{xcolor}
\usepackage{bm}
\usepackage{lmodern}

\newtheorem{theorem}{Theorem}
\newtheorem{lemma}[theorem]{Lemma}

\newtheorem{corollary}[theorem]{Corollary}

\newtheorem{assumption}[theorem]{Assumption}
\newtheorem{remark}[theorem]{Remark}

\newcommand{\C}{\mathbb{C}}
\newcommand{\R}{\mathbb{R}}
\newcommand{\NN}{\mathbb{N}}
\newcommand{\X}{\mathbf{X}}

\newcommand{\x}{\mathbf{x}}
\newcommand{\y}{\mathbf{y}}
\newcommand{\ab}{\mathbf{a}}
\newcommand{\bb}{\mathbf{b}}
\newcommand{\cb}{\mathbf{c}}
\newcommand{\db}{\mathbf{d}}
\DeclareMathOperator{\di}{d}
\newcommand{\ib}{\mathbf{i}}
\DeclareMathOperator{\lb}{lb}
\DeclareMathOperator{\st}{s.t.}
\DeclareMathOperator{\re}{Re}
\DeclareMathOperator{\im}{Im}

\DeclareMathOperator{\opt}{opt}
\DeclareMathOperator{\Kprod}{\mathbf{K}}
\DeclareMathOperator{\kerprod}{K}
\DeclareMathOperator{\Ks}{\mathbf{C}}
\DeclareMathOperator{\kers}{C}
\DeclareMathOperator{\Tr}{Tr}
\DeclareMathOperator{\myspan}{span}
\newcommand{\bB}{\mathbb{B}}

\newcommand{\cQ}{\mathcal{T}}
\newcommand{\quantum}{\mathcal{Q}}
\newcommand{\qu}{u}
\newcommand{\qv}{v}
\newcommand{\qt}{e}
\newcommand{\cP}{\mathcal{P}}
\newcommand{\cM}{\mathcal{M}}
\newcommand{\cG}{\mathcal{G}}
\newcommand{\primal}{\mathfrak{p}}

\providecommand{\keywords}[1]
{
  \small	
  \textbf{{Keywords.}} #1
}
\usepackage{braket}
\usepackage{hyperref}
\usepackage{cleveref}

\title{\bf Convergence rates for polynomial optimization on set products}
\date{\today}
\author{Victor Magron\thanks{{\tt{vmagron@laas.fr}} LAAS-CNRS \& Institute of Mathematics from Toulouse, France}}

\begin{document}
\maketitle
\begin{abstract}
We consider polynomial optimization problems on Cartesian products of basic compact semialgebraic sets. 
The solution of such problems can be approximated as closely as desired by hierarchies of semidefinite programming relaxations, based on classical sums of squares certificates due to Putinar and Schm\"udgen.  
When the feasible set is the bi-sphere, i.e., the Cartesian product of two unit spheres, we show that the hierarchies based on the {Schm\"udgen/Putinar-type} certificates converge to the global minimum of the objective polynomial at a rate in $O(1/t^2)$, where $t$ is the relaxation order. 
Our proof is based on the polynomial kernel method. 
We extend this result to arbitrary sphere products and give a general recipe to obtain convergence rates for polynomial optimization over products of distinct sets. 
Eventually, we rely on our results for the bi-sphere to analyze the speed of convergence of a semidefinite programming hierarchy approximating the order $2$ quantum Wasserstein distance. \newline
~\\
\keywords{
polynomial optimization, 
Cartesian product, 
Moment-SOS hierarchy, 
Positivstellensatz, 
polynomial kernel method, 
quantum Wassertein distance
}
\end{abstract}
\section{Introduction}
\label{sec:intro}
Given a tuple $\x=(x_1,\dots,x_n)$, we denote by $\R[\x]$ the vector space of polynomials with real coefficients. 
Given $r \in \NN$ and a finite set of polynomials $q, g_1,\dots,g_r \subset \R[\x]$, we consider the problem of minimizing $q$ over the \textit{basic compact semialgebraic set}
$$\X:=\{\x \in \R[\x] : g_1(\x) \geq 0, \dots, g_r(\x) \geq 0 \},$$
namely to compute $q_{\min} := \inf_{\x \in \X} q(\x)$.  Note that this polynomial optimization problem is equivalent to
\begin{align}
\label{eq:pop}
q_{\min} = \sup_{\tau \in \R} \{\tau : q - \tau \in \cP_+(\X)  \},
\end{align}
where $\cP_+(\X)$ is the set of polynomials being nonnegative on $\X$. 
Polynomial optimization problems of the form \eqref{eq:pop} capture difficult instances such as the NP-hard MaxCut problem. 
We refer the interested reader to the recent monographs \cite{henrion2020moment,nie2023moment,sparsebook} that address several theoretical and applicative aspects of polynomial optimization. 

Since the characterization of $\cP_+(\X)$ is a highly non-trivial task, there have been substantial research efforts to provide tractable inner approximations, based on sums of squares decompositions. 
A polynomial $p \in \R[\x]$ is called a \textit{sum of squares} if it can be decomposed as $p= \sum_{i} q_{i}^2$, for some $q_{i} \in \R[\x]$. 
The set of sums of squares of polynomials is denoted by $\Sigma[\x]$. 

Let us define $g_0:=1$, $[r] := \{1,\dots,r\}$ and $g_J:= \displaystyle\prod_{j \in J} g_j$, for all $J \subseteq [r]$. 

For instance one can approximate $\cP_+(\X)$ from the inside while relying either on the 
\textit{preordering} $\cQ(\X) \subseteq \cP_+(\X)$ or the \text{quadratic module} $\mathcal{Q}(\X) \subseteq \cP_+(\X)$,   defined as
\begin{align}
\label{eq:preordering}
\cQ(\X) & := \left\lbrace \sum_{J \in [r]} \sigma_J g_J : \sigma_J \in \Sigma[\x]  \right\rbrace, \\
\mathcal{Q}(\X) & := \left\lbrace \sum_{j=0}^m \sigma_j g_j : \sigma_j \in \Sigma[\x]  \right\rbrace. 
\end{align}
When $\X$ is compact, Schm\"udgen's Positivstellensatz states that every polynomial positive on $\X$ belongs to $\cQ(\X)$:
\begin{theorem}[Schm\"udgen's Positivstellensatz {\cite{Schmudgen1991}}]
\label{th:schmudgen}
Let $\X$ be a basic compact semialgebraic set. 
Then for any polynomial $q \in \cP_+(\X)$ and $\varepsilon > 0$, one has $q + \varepsilon \in \cQ(\X)$. 
\end{theorem}
Putinar's Positivstellensatz allows one to replace $\cQ(\X)$ by $\mathcal{Q}(\X)$ in the above result under an assumption slightly stronger than compactness. 
Given a vector $\x \in \R^n$, its Euclidean norm is denoted by $\|\x\|:= \sqrt{\sum_{i=1}^n x_i^2}$. 
The quadratic module $\mathcal{Q}(\X)$ is called \textit{Archimedean} if $N - \|\x\|^2 \in \mathcal{Q}(\X)$, {for some $N >0$}.
\begin{theorem}[Putinar's Positivstellensatz {\cite{putinar1993positive}}]
\label{th:putinar}
Let $\X$ be a basic semialgebraic set and assume that $\mathcal{Q}(\X)$ is Archimedean. 
Then for any polynomial $q \in \cP_+(\X)$ and $\varepsilon > 0$, one has $q + \varepsilon \in \mathcal{Q}(\X)$. 
\end{theorem}
\subsection{Hierarchies of lower bounds for polynomial optimization}
\label{sec:hierarchies}
For a given integer $t \geq \max_{1 \leq j \leq r} \{ \lceil \deg g_j / 2 \rceil \}$, if one considers the truncated preordering or quadratic module defined as
\begin{align}
\label{eq:truncated}
\cQ(\X)_{2t} & := \left\lbrace \sum_{J \in [r]} \sigma_J g_J : \sigma_J \in \Sigma[\x], \quad \deg (\sigma_J g_J) \leq 2t  \right\rbrace,\\
\mathcal{Q}(\X)_{2t} & := \left\lbrace \sum_{j=0}^m \sigma_j g_j : \sigma_j \in \Sigma[\x], \quad \deg (\sigma_j g_j) \leq 2t   \right\rbrace,
\end{align} 
then checking that a polynomial belongs to either $\cQ(\X)_{2t}$ or $\mathcal{Q}(\X)_{2t}$ can be done by solving a \textit{semidefinite program}. 
Semidefinite programming is a particular class of convex programming, that consists of minimizing a linear function under linear matrix inequality constraints  \cite{vandenberghe1996semidefinite}, and for which dedicated solvers can be readily used. 
%

{The hierarchies of Schm\"udgen-type and Putinar-type lower bounds defined as}
\begin{align}
\label{eq:sos}
\lb(q, \cQ(\X))_t := \sup_{\tau \in \R} \{\tau : q - \tau \in \cQ(\X)_{2t}  \}, \\
\lb(q, \mathcal{Q}(\X))_t := \sup_{\tau \in \R} \{\tau : q - \tau \in \mathcal{Q}(\X)_{2t}  \}, \label{eq:putinar}
\end{align}
both converge to $q_{\min}$ when $t$ goes to infinity, as a consequence of Schm\"udgen's and Putinar's Positivstellensatz, respectively. 
{The convergence proof for Putinar-type bounds is provided in \cite{lasserre2001global}.
Let us warn the reader that the lower bounds are indexed by the subscript $t$ whenever the truncation order of the preordering or quadratic module is equal to $2 t$. 
When providing convergence rates for the bi-sphere later on, the lower bounds will be indexed by $2t$ and the truncation order of the preordering or quadratic module will be equal to $4 t$. 
}
Note that when $\X$ involves either equality constraints only or a single inequality constraint, then the  preordering and quadratic module coincide. 
However, only the former remains true when one considers Cartesian products. 
\subsection{Related works on convergence rates}
\label{sec:related}
Recently, there has been a lot of dedicated research efforts in analyzing the asymptotic behavior of the hierarchies of Schm\"udgen-type and Putinar-type bounds defined in \eqref{eq:sos}-\eqref{eq:putinar}, respectively. 
For a polynomial $q$ achieving its minimal value $q_{\min}$ over a general compact semialgebraic set $\X$, it was proved in \cite{schweighofer2004complexity} that the Schm\"udgen-type bounds $\lb(q, \cQ(\X))_t$ converge to $q_{\min}$ at a rate in $O(1/t^c)$, where $c$ is a constant depending on $\X$. 
If $\mathcal{Q}(\X)$ is Archimedean, an initial result  provided in \cite{nie2007complexity} shows that the Putinar-type bounds $\lb(q, \mathcal{Q}(\X))_t$ converge to $q_{\min}$ at a rate in $O(1/\log(t)^c)$. 
Recently, this rate has been exponentially improved in \cite{baldi2023effective}, where the authors prove that the Putinar-type bounds converge to $q_{\min}$ at a rate in $O(1/t^c)$, thus matching the convergence rate of Schm\"udgen-type bounds. 
Without any compactness assumption on the feasible set, one can rely on Putinar-type hierarchies based on sums of squares of rational fractions with (prescribed) uniform denominators. 
For these hierarchies, a similar rate in $O(1/t^c)$ has been obtained in \cite{mai2022complexity}. 

For special cases of feasible sets including the unit sphere/ball, the standard simplex and the hypercube, the value of the constant $c$ can be explicitly given.  
An improved convergence rate in $O(1/t^2)$ for Schm\"udgen-type bounds has been obtained 
in \cite{fang2021sum} for the unit sphere,  
in \cite{slot2022sum} for the unit ball and the standard simplex, and 
in \cite{laurent2023effective} for the hypercube. 
The case of the binary hypercube $\{0,1\}^n$ is covered in \cite{slot2023sum}. 
For Putinar-type bounds, a rate in $O(1/t)$ has been recently obtained in \cite{baldi2024degree} for the hypercube. 
For polynomial optimization with correlative sparsity, stronger rates have been obtained in  \cite{korda2025convergence} under the assumption that the input data are (sufficiently) sparse. 

In addition, convergence rates are available for the \textit{generalized moment problem} (GMP), which consists of minimizing a linear function over the cone of positive Borel measures supported on $\R^n$. 
Note that polynomial optimization can be cast as a special instance of the GMP, and that Schm\"udgen-type and Putinar-type hierarchies can be naturally derived for other GMP instances. 
When the GMP instance involves finitely many constraints, one can obtain a rate similar to the one for polynomial optimization on the measure support  \cite{gamertsfelder2025effective}. 
Specialized rates have been obtained for certain GMP instances involving infinitely (countably) many constraints, in particular for optimal control  \cite{korda2017convergence}, volume estimation  \cite{korda2018convergence}, and exit location estimation of stochastic processes \cite{schlosser2024specialized}. 

We refer to the recent survey \cite{laurent2024overview} 
for a more detailed overview of convergence rates and related aspects, in particular the tightness of the performance analysis, the exponential convergence under local optimality conditions, and finite convergence. 
\subsection{Our contributions}
\label{sec:contrib}
In this work, we mainly focus on hierarchies of Schm\"udgen-type bounds for polynomial optimization over Cartesian products of compact semialgebraic sets. 

We first provide in Section \ref{sec:proofbisphere} a convergence rate of Schm\"udgen-type, {or equivalently Putinar-type}, bounds when minimizing a polynomial on the  product of spheres $\X = S^{n-1} \times S^{n-1}$. 
Given two $n$-dimensional tuples of real variables $\x=(x_1,\dots,x_n)$ and $\y=(y_1,\dots,y_n)$, and a polynomial $q \in \R[\x,\y]$, we define $q_{\min}:= \min_{(\x,\y) \in \X} q(\x,\y) = \min_{\x,\y \in S^{n-1}} q(\x,\y)$ and $q_{\max} :=  \max_{(\x,\y) \in \X} q(\x,\y)$. 
\begin{theorem}
\label{th:bisphererate}
Let $\X = S^{n-1} \times S^{n-1}$ be the Cartesian product of two unit spheres, and let $q \in \R[\x,\y]$ be a polynomial of degree $d$. 
Then for any $t \geq 2 n d \sqrt{d}$, the lower bound ${\lb(q, \mathcal{Q}(\X))_{2t}}$ for the minimization of $q$ over $\X$ satisfies:
\begin{align}
\label{eq:bisphererate}
q_{\min} - {\lb(q, \mathcal{Q}(\X))_{2t}} \leq \frac{C_{\X}(n,d)}{t^2} \cdot (q_{\max} - q_{\min}). 
\end{align}
Here, $C_{\X}(n,d)$ is a constant depending only on $n$ and $d$. 
This constant has a polynomial dependence on $n$ at fixed $d$, and a polynomial dependence on $d$ at fixed $n$. See relation \eqref{eq:cX} for details. 
\end{theorem}
The proof of Theorem \ref{th:bisphererate} is given in Section \ref{sec:proofbisphere}. 
{
\begin{remark}
\label{rk:bisphererate}
One can replace $2 t$ by $t$ in Theorem \ref{th:bisphererate} to obtain that for any $t \geq 4 n d \sqrt{d}$, the lower bound $\lb(q, \mathcal{Q}(\X))_{t}$ satisfies:
\begin{align}
\label{eq:bisphereratet}
q_{\min} - \lb(q, \mathcal{Q}(\X))_{t} \leq 4 \frac{C_{\X}(n,d)}{t^2} \cdot (q_{\max} - q_{\min}). 
\end{align}
\end{remark}
}

We extend this result in Section \ref{sec:multisphere} (Theorem \ref{th:multisphererate}) to the case of arbitrary sphere products, and in Section \ref{sec:general} (Theorem \ref{th:generalrate}) to the case of products of distinct sets. 
{
As proved in \cite[Theorem 3]{schweighofer2004complexity}, the hierarchy of Schm\"udgen-type relaxations has the convergence rate $O(1/t^c)$ for a general compact set $\X$, where the constant $c$ depends on $\X$.
{
We emphasize that the result from \cite{schweighofer2004complexity} only yields the existence of such a constant $c$ but cannot directly provide the value, e.g., $c=2$ for specific constraint sets. 
The drawback of \cite[Theorem 3]{schweighofer2004complexity} is that $c$ depends on the description of the constraint set in an unspecified way. 
For any concrete situation, such as the sphere, the unit ball, the unit cube or the standard simplex, one could  hope to extract a suitable $c$ from the proof of Section 3 from \cite{schweighofer2004complexity}, see \cite[Remark 10]{schweighofer2004complexity} for more explanation.  
It would in particular require to estimate intermediate constants satisfying the {\L}ojasiewicz inequality used in \cite{schweighofer2004complexity}, which appears to be delicate. 
} 
Theorem \ref{th:multisphererate} implies that one can select $c=2$ when $\X$ is the product of two or more spheres. 
}

Eventually we analyze in Section \ref{sec:rategmp} the convergence rate of a problem arising from quantum information theory, called the quantum Wasserstein distance. 
Computing this distance boils down to solving a specific instance of the generalized moment problem on the complex bi-sphere. 
We provide a rate in $O(1/t^2)$, stated in Theorem \ref{th:qwratereal}, by first changing variables from complex to real, then by combining Theorem \ref{th:bisphererate} with the recent convergence rate results from \cite{gamertsfelder2025effective}, obtained for the generalized moment problem. 

After the initial submission of this work in May 2025, we became aware of the related work \cite{tran2025convergence}, that was submitted as an arXiv preprint two months later in July 2025. 
In \cite[Theorem 3.3]{tran2025convergence} (see also Corollary 3.4), the authors provide similar convergence rates with a proof also based on polynomial kernel products. 
The authors of \cite{tran2025convergence} also provide convergence rate results for the Hausdorff distance between (truncated) pseudo-moment and moment sequences; see Section 3.2 and Section 4 therein. 
We emphasize that the two works were carried out independently. 

\subsection{Outline of the proof technique}
\label{sec:outline}
For a given compact semialgebraic set $\X \subset \R^{n}$ and a polynomial $p \in \R[\x]$, we denote the supremum-norm of $p$ on $\X$ by $\|p\|_{\X}:= \max_{\x \in \X} |p(\x)|$. 
To obtain the convergence rate of Theorem \ref{th:bisphererate}, we follow the exact same technical strategy pursued in \cite{fang2021sum,slot2022sum} for the case of the unit sphere $S^{n-1}$, the unit ball and the standard simplex, as well as in \cite{laurent2023effective} for the case of the hypercube. 
Let us denote by $\R[\x,\y]_d$ the vector space of polynomials of degree at most $d$. 
Similarly let $\cP_+( \X)_d$ be the subset of polynomials from $\R[\x,\y]_d$ that are nonnegative on $\X$. 
Let $q \in \R[\x,\y]_d$  and $\X = S^{n-1} \times S^{n-1}$. 
The goal is to prove that 
\[
{q - q_{\min} + \varepsilon \in \mathcal{Q}(\X)_{4t}  (=\cQ(\X)_{4 t}) }
\]
for some small $\varepsilon$. 
Up to translation and scaling, we assume that $q_{\min}=0$ and $q_{\max}=1$. 
We will construct an invertible linear operator $\Kprod : \R[\x,\y]_d \to \R[\x,\y]_d$ which satisfies the following properties:
\begin{align}
\label{eq:poneref} \Kprod(1)=1, \tag{P1}\\
\label{eq:ptworef} \Kprod p \in {\mathcal{Q}(\X)}_{4 t}, \quad \text{for all } p \in \cP_+( \X)_d,  \tag{P2} \\
\label{eq:pthreeref} \| \Kprod^{-1} q - q \|_{\X} \leq \varepsilon. \tag{P3}
\end{align}
Then it follows from \cite[Lemma 7]{slot2022sum} that $q + \varepsilon \in {\mathcal{Q}(\X)}_{4 t}$ and thus $q_{\min} - \lb(q, {\mathcal{Q}(\X)})_{2t} \leq \varepsilon$. 
As in \cite{fang2021sum,slot2022sum,laurent2023effective}, the statement of Theorem \ref{th:bisphererate} is proven by showing the existence, for each large enough $t$, of an operator $\Kprod$ satisfying \eqref{eq:poneref}, \eqref{eq:ptworef} and \eqref{eq:pthreeref} with $\varepsilon = O(1/t^2)$. 

The way to construct such an operator is to follow the so-called \textit{polynomial kernel method}. 
In Section \ref{sec:ratesphere}, we explain how such a kernel is found in \cite{fang2021sum} for the case of the unit sphere. 
Our rate for the bi-sphere is obtained by multiplying the two kernels associated to each unit sphere. 
This idea of kernel multiplication is the same as the one previously used for the hypercube in \cite{laurent2023effective} (see also \cite{korda2025convergence}). 

\section{Preliminaries: convergence rate on the unit sphere}
\label{sec:ratesphere}
For a multi-index $\alpha = (\alpha_1,\dots,\alpha_n) \in \NN^n$, we use the notation $|\alpha|=\alpha_1 + \dots + \alpha_n$. 
Given $d \in \NN$, let $\NN_d^n := \{\alpha \in \NN^n : |\alpha| \leq d \}$. 
The dimension of $\NN_d^n$ is equal to the number of monomials of degree at most $d$, i.e., $\binom{n+d}{d}$. \\ 
A map $\kers : S^{n-1} \times S^{n-1} : \R$ on $S^{n-1}$ is called a polynomial kernel when $\kers(\x,\x')$ is a polynomial in the variables $\x$, $\x'$. 
{For each $k$, let $H_k[\x] := \myspan \{ P_{\alpha} : \alpha \in \NN^n, |\alpha|=k \}$ be the spanning set of orthonormal \textit{spherical harmonics} of degree $k$, i.e., homogeneous polynomials of degree $k$ lying in the kernel of the Laplace operator. }
Let us equip $S^{n-1}$ with the uniform Haar measure $\mu$, and consider the \textit{Christoffel-Darboux} kernel $\kers_{2 t}$, which is defined in terms of {the} orthonormal basis $\{P_{\alpha} : \alpha \in \NN^n\}$ for $\R[\x]$ with respect to $(S^{n-1},\mu)$ as:
\begin{align}
\label{eq:cdksphere}
\kers_{2t}(\x,\x') := \sum_{k=0}^{2t} \kers^{(k)} (\x,\x'), \quad \text{where } \kers^{(k)} (\x,\x') := \sum_{|\alpha|=k} P_{\alpha}(\x) P_{\alpha}(\x').
\end{align}
Then one associates to $\kers_{2t}$ an operator $\Ks_{2t}: \R[\x] \to \R[\x]$, defined as follows:
\begin{align}
\label{eq:ksphere}
\Ks_{2t} p(\x) := \int_{S^{n-1}} \kers_{2t}(\x,\x') p(\x')  \di  \mu(\x') .
\end{align}
This operator $\Ks_{2t}$ is called \textit{reproducing} as it satisfies $\Ks_{2t} p = p$, for all $p \in \R[\x]_{2t}$. Therefore, all its eigenvalues are equal to $1$. 

One has $p(\x) = \sum_{k=0}^d p_k(\x)$, for all $\x \in S^{n-1}$, with $p_k \in H_k$ given by 
\[
p_k = \int_{S^{n-1}} \kers^{(k)} (\x,\x') p(\x') \di  \mu(\x') .
\]

In this case one has the following closed form of the Christoffel-Darboux kernel 
\begin{align}
\label{eq:funkhecke} 
\kers_{2t}(\x,\x') = \sum_{k=0}^{2t} \cG_k^{(n-1)} (\x \cdot \x'), \quad \text{for all } \x,\x' \in S^{n-1}.
\end{align}
In the above formula, the set $\{\cG_k^{(n)} : k \in \NN \}$ of \textit{Gegenbauer polynomials} are defined, for all $n \geq 2$, as the set of orthogonal polynomials with respect to the weight function $w_n(x):= \hat{w}_n (1-x^2)^{\frac{n-2}{2}}$, where $\hat{w}_n$ is a normalization constant ensuring that $\int_{-1}^1 w_n(x)  \di  x = 1$. 
We refer the interested reader to \cite{szego1975orthogonal} for a detailed survey about univariate orthogonal polynomials. 
The formula \eqref{eq:funkhecke} is also known as the \textit{Funk-Hecke} formula, stated for instance in \cite[Theorem A.34]{rubin2015introduction}. 

The next step is to perturb the eigenvalues of the reproducing operator $\Ks_{2t}$ by considering the following kernel instead:
\begin{align}
\label{eq:cdksphereperturb}
\kers_{2t}(\x,\x'; \lambda) := \sum_{k=0}^{2t} \lambda_k \kers^{(k)} (\x,\x') = \sum_{k=0}^{2t} \lambda_k \cG_k^{(n-1)} (\x \cdot \x'), \quad \text{for all } \x,\x' \in S^{n-1}. 
\end{align}
Then $\lambda = (\lambda_k)$ is the set of eigenvalues of the operator obtained in \eqref{eq:ksphere} by considering $\kers_{2t}(\x,\x'; \lambda)$ instead of $\kers_{2t}(\x,\x')$. 
The following result provides one way to prove that $\Ks_{2t}$ satisfies \eqref{eq:ptworef}. 
\begin{lemma}[See {\cite[Lemma 11]{slot2022sum}}]
\label{lemma:ptwo}
Let $\X$ be a compact semialgebraic set, and let $\mu$ be a finite measure supported on $\X$. 
Let $Q$ be a convex cone, and suppose that $\kers : \X \times \X \to \R$ is a polynomial kernel for which $\kers(\cdot,\x') \in Q$ for each $\x' \in \X$ fixed. 
Then if $p \in \R[\x]$ is nonnegative on $\X$, we have $\Ks p \in Q$. 
\end{lemma}
{
\begin{remark}
\label{rk:ptwo}
Let $\X$, $\mu$ be as in Lemma \ref{lemma:ptwo}. 
When selecting $Q=\cQ(\X)_{2t}$ in Lemma \ref{lemma:ptwo}, the operator $\Ks$ satisfies \eqref{eq:ptworef} (with $\cQ(\X)_{2t}$ instead of $\mathcal{Q}(\X)_{4t}$). 
If $\X$ is the unit sphere, one can equivalently select $Q = \mathcal{Q}(\X)_{2t}$. 
\end{remark}
}
The proof of Lemma \ref{lemma:ptwo} relies on Tchakaloff's Theorem \cite{tchakaloff1957cubature} and the exisence of cubature rules. 
Now coming back to the case of the unit sphere, assume that $\lambda = (\lambda_k)$ is the set of coefficients of a univariate sum of squares in the basis of Gegenbauer polynomials. 
Then at each fixed $\x'$, the polynomial $\kers_{2t}(\x,\x'; \lambda)$, defined in \eqref{eq:cdksphereperturb}, is equal to a sum of squares of degree at most $2t$ on $\X$, so $\kers_{2t}(\x,\x'; \lambda) \in  {\mathcal{Q}(\X)_{2t}}$ and Lemma \ref{lemma:ptwo} implies that $\Ks_{2t}$ satisfies \eqref{eq:ptworef}. 

Next, if the harmonic decomposition of $q$ is $q(\x)=\sum_{k=0}^d q_k(\x)$ for all $\x \in S^{n-1}$, with $q_k \in H_k[\x]$, then one has $\Ks_{2t} q = \sum_{k=0}^d \lambda_k q_k$. 
If in addition, $\lambda_0=1$ and $\lambda_k \neq 0$ for all $k$ then $\Ks_{2t}^{-1} q = \sum_{k=0}^d 1/\lambda_k q_k$. 
By \cite[Lemma 14]{slot2022sum}, the resulting perturbed operator satisfies \eqref{eq:poneref} if $\lambda_0=1$, and it satisfies \eqref{eq:pthreeref} with 
\[
\varepsilon = \max_{1\leq k \leq d} \max_{\|\x\|=1} |q_k(\x)| \cdot \left|1 - 1/\lambda_k \right|.
\]
It remains to give an upper bound of the two multiplied quantities. 
For the first quantity, one can define the so-called \textit{harmonic constant}:
\begin{align}
\label{eq:harmonic}
\gamma(S^{n-1})_d := \max_{p \in \R[\x]_d} \max_{0\leq k \leq d} \frac{\|p_k\|_{S^{n-1}}}{\|p\|_{S^{n-1}}}.
\end{align}
In case of minimizing a homogeneous polynomial of degree $d$, one can define a similar constant denoted by $\gamma(S^{n-1})_{=d}$. 
This latter constant can be bounded independently of the dimension $n$ as a consequence of \cite[Proposition 5]{fang2021sum}. 
More generally, one has the following result. 
\begin{lemma}
\label{lemma:harmonic}
For all integers $d, n$, {with $n \geq 3$}, the constant $\gamma(S^{n-1})_d$ depends polynomially on $n$ (when $d$ is fixed) and polynomially on $d$ (when $n$ is fixed). One has
\begin{align}
\label{eq:harmonicbound}
\gamma(S^{n-1})^2_d \leq \max_{0\leq k\leq d} \cG_k^{n}(1) = \max_{0\leq k \leq d} \left(1+\frac{2k}{n-2}\right) \cdot \binom{k+n-3}{k}. 
\end{align}
\end{lemma}
\begin{proof}
As in the proof of \cite[Proposition 15]{slot2022sum}, one relies on the reproducing properties of the Christoffel-Darboux kernel to show that 
\[
\gamma(S^{n-1})^2_d \leq \max_{0\leq k \leq d} \max_{\x \in S^{n-1}} \kers^{(k)} (\x,\x) = \max_{0 \leq k \leq d} \cG_k^{(n-1)}(1). 
\]
The desired expression then follows from the value of the Gegenbauer polynomials at $1$, see, e.g., (2.9) in  \cite{xu1998summability}. 
\end{proof}
To finally obtain the desired convergence rate, the second quantity can be bounded as follows. 
\begin{lemma}[See {\cite[Theorem 6]{fang2021sum}}]
\label{lemma:boundlambda}
Let $n,d \in \NN$. Then for every $t \geq 2 n d \sqrt{d}$ there exists a univariate sum of squares $\sigma(x)=\sum_{k=0}^{2t} \lambda_k \cG_k^{(n)}(x)$ of degree $2t$ with $\lambda_0=1$, $1/2 \leq \lambda_k \leq 1$, for all $1 \leq k \leq d$, and 
\begin{align}
\label{eq:boundlambda}
\sum_{k=1}^{d} (1 - \lambda_k ) &\leq \frac{ n^2 d^3}{t^2}, \\
\sum_{k=1}^{d} \left|1 - 1/\lambda_k \right| \leq 2 \sum_{k=1}^{d} (1 - \lambda_k ) & \leq \frac{2 n^2 d^3}{t^2} \nonumber.
\end{align}
\end{lemma}
\section{Convergence rate for polynomial optimization on set products}
\label{sec:ratepop}
\subsection{Products of two spheres}
\label{sec:proofbisphere}
This section is dedicated to the proof of Theorem \ref{th:bisphererate}. 
\begin{proof}
Up to translation and scaling, we assume that $q_{\min}=0$ and $q_{\max}=1$. 
Let us equip $\X = S^{n-1} \times S^{n-1}$ with the product measure $\mu_{\X} := \mu \otimes \mu$, where $\mu$ is the uniform Haar measure on $S^{n-1}$. 
Given the orthonormal basis $\{P_{\alpha} : \alpha \in \NN^n\}$ for $\R[\x]$ (respectively for $\R[\y]$) with respect to $(S^{n-1},\mu)$,  $\{P_{\alpha}(\x) P_{\beta}(\y)  : \alpha, \beta \in \NN^n\}$ is an orthonormal basis for $\R[\x,\y]$ with respect to $(\X,\mu_{\X})$. 
With $t \geq 2 n d \sqrt{d}$, and $\lambda=(\lambda_k)$ as in Lemma \ref{lemma:boundlambda}, we build the desired operator $\Kprod_{2t} : \R[\x,\y] \to \R[\x, \y]$ as follows:
\begin{align}
\label{eq:cdk}
\Kprod_{2t} q(\x,\y) & := \int_{\X} \kerprod_{2t}((\x,\y),(\x',\y'); \lambda) q(\x',\y')  \di  \mu_{\X}(\x',\y'), \\
\text{where }\kerprod_{2t}((\x,\y),(\x',\y'); \lambda) & := \sum_{k,s=0}^{2t} \lambda_k \lambda_s \kerprod^{(k,s)} ((\x,\y),(\x',\y')),\\ 
\quad \text{and } \kerprod^{(k,s)} ((\x,\y),(\x',\y')) & := \sum_{|\alpha|=k, |\beta|=s} P_{\alpha}(\x) P_{\beta}(\y) P_{\alpha}(\x') P_{\beta}(\y').
\end{align} 
Let us show that $\Kprod_{2t}$ satisfies the three desired properties \eqref{eq:poneref}, \eqref{eq:ptworef} and \eqref{eq:pthreeref}. 

For each $k,s$, one defines $H_{k,s}[\x,\y] := H_k[\x] \otimes H_s[\y] = \myspan \{P_{\alpha}(\x) P_{\beta}(\y)  : \alpha, \beta \in \NN^n, |\alpha|=k, |\beta|=s \}$.

Then every $q \in \R[\x,\y]_d$ has a decomposition 
\begin{align}
\label{eq:qharmonic}
q(\x,\y) = \sum_{k+s \leq d} q_{k,s}(\x,\y), \text{ for all } (\x,\y) \in \X, 
\end{align}
with $q_{k,s} \in H_{k,s}[\x,\y]$, given by
\begin{align}
\label{eq:qks}
q_{k,s} = \int_{\X} \kers^{(k)} (\x,\x') \kers^{(s)} (\y,\y') q(\x',\y') \di  \mu(\x') .
\end{align}
By construction, one has 
\begin{align*}
\kerprod^{(k,s)} ((\x,\y),(\x',\y')) & = \kers^{(k)} (\x,\x') \kers^{(s)} (\y,\y'), \\
\kerprod_{2t} ((\x,\y),(\x',\y'); \lambda) & = \sum_{k,s=0}^{2t} \lambda_k \lambda_s \kers^{(k)} (\x,\x') \kers^{(s)} (\y,\y') = \kers_{2t} (\x,\x'; \lambda) \kers_{2t} (\y,\y'; \lambda), 
\end{align*}
where $\kers_{2t} (\cdot,\cdot; \lambda)$ has been defined in \eqref{eq:cdksphereperturb}. 
Then one has
\begin{align}
\label{eq:diagprod}
\Kprod_{2t} q = \sum_{k+s\leq d} \lambda_k \lambda_s q_{k,s}.
\end{align}
Since $\lambda_0=1$, \eqref{eq:poneref} is satisfied. 

To prove that \eqref{eq:ptworef} is satisfied, recall that 
\begin{align*}
\kers_{2t}(\x,\x'; \lambda) := \sum_{k=0}^{2t} \lambda_k \kers^{(k)} (\x,\x') = \sum_{k=0}^{2t} \lambda_k \cG_k^{(n-1)} (\x \cdot \x'), \quad \text{for all } \x,\x' \in S^{n-1}, \\
\kers_{2t}(\y,\y'; \lambda) := \sum_{s=0}^{2t} \lambda_s \kers^{(s)} (\y,\y') = \sum_{s=0}^{2t} \lambda_s \cG_s^{(n-1)} (\y \cdot \y'), \quad \text{for all } \y,\y' \in S^{n-1}, 
\end{align*}
and that we assumed that $\lambda=(\lambda_k)$ is the set of coefficients of a univariate sum of squares in the basis of Gegenbauer polynomials. \\
Then at each fixed $(\x',\y')$, the polynomial  $\kerprod_{2t} ((\x,\y),(\x',\y'); \lambda)$ is equal to a product of sums of squares of degree at most $4 t$ on $\X$, so $\kerprod_{2t} ((\x,\y),(\x',\y'); \lambda) \in {\mathcal{Q}(\X)}_{4t}$ and Lemma \ref{lemma:ptwo} (when selecting $Q=\cQ(\X)_{4t} {=\mathcal{Q}(\X)_{4t}} $) implies that $\Kprod_{2t}$ satisfies \eqref{eq:ptworef}. 
Note that the degree $4 t$ of this representation is twice larger than the one obtained in the case of the unit sphere.   
%

We end the proof by showing that the operator satisfies \eqref{eq:pthreeref}, namely, 
\begin{align}
\label{eq:pthree}
\max_{(\x,\y) \in \X} | \Kprod^{-1} q(\x,\y) - q(\x,\y)| \leq \varepsilon = \frac{C_{\X}(n,d)}{t^2},
\end{align}
with 
\begin{align}
\label{eq:cX}
C_{\X}(n,d) := 8 n^2 d^3 \binom{d+2}{2} \cdot \gamma(S^{n-1})_d^2,
\end{align}
and $\gamma(S^{n-1})_d$ is the constant defined in  \eqref{eq:harmonic}. 
By \eqref{eq:diagprod}, one has 
\begin{align}
\label{eq:pthreeprod}
\Kprod_{2t}^{-1} q - q = \sum_{1\leq k+s\leq d}  \left( \frac{1}{\lambda_k \lambda_s} -1 \right) q_{k,s}.
\end{align}
To bound this quantity over $\X$, we follow a strategy similar to the one from \cite{laurent2023effective}, that relies on Bernoulli's identity: for any $x \in [0, 1]$ and $m \geq 1$, we have $1-(1-x)^m \leq m x$. 
For all $1 \leq k \leq d$, let us define $\eta_k:=1-\lambda_k$, and $\eta := \max_{1 \leq k \leq d} \eta_k$. 
Since we selected $\lambda$ as in Lemma \ref{lemma:boundlambda}, one has $1/2 \leq \lambda_k \leq 1$, for all $1 \leq k \leq d$, then $0 \leq \eta \leq 1/2$, so one can apply Bernoulli's identity to obtain
\begin{align}
\label{eq:bernoulli}
1 - \lambda_k \lambda_s = 1 - (1 - \eta_k) (1 - \eta_s) \leq 1 - {(1 - \eta)^2} \leq 2 \eta. 
\end{align}
By \eqref{eq:boundlambda}, one has $1 - \lambda_k \leq \sum_{k=1}^d (1 - \lambda_k) \leq \frac{ n^2 d^3}{t^2}$, for all $1 \leq k \leq d$, implying that $\eta \leq \frac{ n^2 d^3}{t^2}$. 
Next, one has
\begin{align}
\label{eq:fracbound}
\left|1 - \frac{1}{\lambda_k \lambda_s}\right| = \frac{|1 - \lambda_k \lambda_s|}{|\lambda_k \lambda_s|} \leq 4 (1 - \lambda_k \lambda_s). 
\end{align}
Then by combining  \eqref{eq:bernoulli} and \eqref{eq:fracbound}, we obtain 
\[
\sum_{1\leq k+s\leq d}  \left( \frac{1}{\lambda_k \lambda_s} -1 \right) \leq 8 \frac{ n^2 d^3}{t^2} \cdot \left|\NN_d^2\right| = 8 \frac{ n^2 d^3}{t^2}  \binom{d+2}{2}.
\]
Eventually let us define $p_{k}(\x,\y):=\sum_{s} q_{k,s} (\x,\y)$. 
Let us fix $\y \in S^{n-1}$.  
From \eqref{eq:qharmonic} one obtains the harmonic decomposition of $q(\x,\y) = \sum_{k} p_{k} (\x,\y)$, with $p_{k}(\cdot,\y) \in H_k[\x]$. Therefore, one has
\[ 
\|p_{k}(\cdot,\y)\|_{S^{n-1}} \leq \gamma(S^{n-1})_d \cdot \|q(\cdot,\y)\|_{S^{n-1}} \leq \gamma(S^{n-1})_d.
\] 
The above inequality is valid for all $\y \in S^{n-1}$, thus one has $\|p_{k}\|_{\X} \leq \gamma(S^{n-1})_d$. 
Now at each fixed $k$ and $\x$, one has $q_{k,s}(\x,\cdot) \in H_s[\y]$, yielding the harmonic decomposition $p_{k}(\x,\cdot)=\sum_{s} q_{k,s} (\x,\cdot)$ and 
\[
\|q_{k,s}(\x,\cdot)\|_{S^{n-1}} \leq \gamma(S^{n-1})_d \cdot \|p_k(\x,\cdot)\|_{S^{n-1}} \leq \gamma(S^{n-1})_d \cdot \|p_{k}\|_{\X} \leq \gamma(S^{n-1})^2_d. 
\]
Since the above inequality is valid for all $\x \in S^{n-1}$, we obtain $\|q_{k,s}\|_{\X} \leq \gamma(S^{n-1})^2_d$, yielding the desired bound. 
\end{proof}
\begin{remark}
If one assumes that $q$ is bi-homogeneous polynomial of degree $d$, i.e., homogeneous of degree $d$ with respect to $\x$ (at fixed $\y$) and $\y$ (at fixed $\x$), then one can replace $\gamma(S^{n-1})^2_d$ by the quantity $\gamma(S^{n-1})^2_{=d}$ that can be bounded independently of the dimension $n$. 
This follows from \cite[Proposition 5]{fang2021sum} which states that the supremum-norm of the harmonic components of a homogeneous polynomial can be bounded by a constant independent of the dimension. 
\end{remark}
\subsection{Extension to arbitrary sphere products}
\label{sec:multisphere}
One can easily extend the rate given in Theorem \ref{th:bisphererate} when optimizing a polynomial over the Cartesian product of $m$  spheres $\X = S^{n-1} \underbrace{\times \dots \times}_{m-1}  S^{n-1}$. 
\begin{theorem}
\label{th:multisphererate}
For $m \geq 2$, let $\X = S^{n-1} \times \dots \times  S^{n-1}$ be the Cartesian product of $m$ unit spheres, and let $q \in \R[\x_1,\dots,\x_m]_d$. 
Then for any $t \geq 2 n d \sqrt{d}$, the lower bound $\lb(q, {\mathcal{Q}(\X)})_{mt}$ for the minimization of $q$ over $\X$ satisfies:
\begin{align}
\label{eq:multisphererate}
q_{\min} - \lb(q, {\mathcal{Q}(\X)})_{mt} \leq \frac{C_{\X}(m,n,d)}{t^2} \cdot (q_{\max} - q_{\min}),
\end{align}
with 
\begin{align}
\label{eq:cXm}
C_{\X}(n,d,m) := m 2^m n^2 d^3 \binom{d+m}{m} \cdot \gamma(S^{n-1})_d^m,
\end{align}
and $\gamma(S^{n-1})_d$ is the constant defined in  \eqref{eq:harmonic}. 
%
\end{theorem}
\begin{proof}
Up to translation and scaling, we assume that $q_{\min}=0$ and $q_{\max}=1$. 
Let us fix $t \geq 2 n d \sqrt{d}$. 
We outline how this rate is obtained by adapting the proof of Theorem \ref{th:bisphererate} given in Section \ref{sec:proofbisphere}. 
\begin{itemize}
\item We equip $\X$ with the product measure $\mu_{\X}:=\mu \overbrace{\otimes \cdots \otimes}^{m-1} \mu$, where $\mu$ is the uniform Haar measure on $S^{n-1}$. The orthonormal basis with respect to $(\X,\mu_{\X})$ is obtained by an $m$-th order tensorization.
{Given the orthonormal basis $\{P_{\alpha(i)} : \alpha(i) \in \NN^n\}$ for $\R[\x_i]$ with respect to $(S^{n-1},\mu)$,  $\{P_{\alpha(1)}(\x_1) \dots P_{\alpha(m)}(\x_m)  : \alpha(1),\dots,\alpha(m) \in \NN^n\}$ is an orthonormal basis for $\R[\x_1,\dots,\x_m]$ with respect to $(\X,\mu_{\X})$};
\item For each $\kappa \in \NN^m$, one defines $H_{\kappa}[\x_1,\dots,\x_m]$ as the span of products of $m$ orthonormal polynomials that are homogeneous of degrees $\kappa_1$ in $\x_1, \dots, \kappa_m$ in $\x_m$,  respectively. Doing so, every polynomial $q \in \R[\x_1, \dots, \x_m]_d$ has a decomposition $q = \sum_{|\kappa| \leq d} q_{\kappa}$, with $q_{\kappa} \in H_{\kappa}[\x_1,\dots,\x_m]$; 
\item {The perturbed kernel is obtained by multiplying $m$ kernels. Let us note $\x = (\x_1,\dots,\x_m)$ and $\x' = (\x_1',\dots,\x_m')$. 
With $t \geq 2 n d \sqrt{d}$, and $\lambda=(\lambda_k)$ as in Lemma \ref{lemma:boundlambda}, let us define $\lambda_\kappa := \lambda_{\kappa_1} \cdots \lambda_{\kappa_m} $, for all $\kappa \in \NN_t^n$. 
We build the desired operator $\Kprod_{2t} : \R[\x] \to \R[\x]$ as follows:
\begin{align}
\label{eq:cdkmulti}
\Kprod_{2t} q(\x) & := \int_{\X} \kerprod_{2t}(\x,\x'); \lambda) q(\x')  \di  \mu_{\X}(\x'), \\
\text{where }\kerprod_{2t}(\x,\x'); \lambda) & := \sum_{\kappa_1,\dots, \kappa_m=0}^{2t} \lambda_{\kappa} \kerprod^{\kappa} (\x,\x'),\\ 
\quad \text{and } \kerprod^{\kappa} (\x,\x') & := \sum_{|\alpha(1)|=\kappa_1,\dots, |\alpha(m)|=\kappa_m} \prod_{i=1}^m P_{\alpha(i)}(\x_i) P_{\alpha(i)}(\x_i') .
\end{align}
\item Then one proves exactly as for the bi-sphere  that
\begin{align}
\label{eq:diagprodmulti}
\Kprod_{2t} q = \sum_{|\kappa|\leq d} \lambda_{\kappa} q_{\kappa}. 
\end{align}
Since $\lambda_0=1$, \eqref{eq:poneref} is satisfied;
\item The polynomial $\kerprod_{2t}(\x,\x'); \lambda)$ obtained at each fixed second argument belongs to ${\mathcal{Q}(\X)}_{2mt}$ and Lemma \ref{lemma:ptwo} implies that the operator $\Kprod_{2t}$ satisfies \eqref{eq:ptworef} (with $\mathcal{Q}(\X)_{2mt}$ instead of $\mathcal{Q}(\X)_{4t}$)};
\item  To prove \eqref{eq:pthreeref}, one first obtains
\begin{align}
\label{eq:pthreemulti}
\Kprod_{2t}^{-1} q - q = \sum_{1\leq |\kappa|\leq d}  \left( \frac{1}{\lambda_{\kappa}} -1 \right) q_{\kappa}.
\end{align}
For all $1 \leq k \leq d$, let us define $\eta_k:=1-\lambda_k$ and $\eta := \max_{1 \leq k \leq d} \eta_k$. 
As for \eqref{eq:bernoulli}, we rely on Bernoulli's identity to obtain the inequality 
\[
{
1 - \lambda_{\kappa} = 1 - \lambda_{\kappa_1} \cdots \lambda_{\kappa_m} = 1 - (1 - \eta_{\kappa_1}) \cdots (1 - \eta_{\kappa_m}) \leq 1 - (1 - \eta)^m  \leq m \eta,
}
\]
with $\eta \leq \frac{n^2 d^3}{t^2}$. 
Since  $\lambda_k \geq 1/2$, one has $|1 - \lambda_{\kappa}^{-1}| \leq 2^m (1 - \lambda_{\kappa})$. 
Eventually, one obtains $\|q_{\kappa}\|_{\X} \leq  \gamma(S^{n-1})^m_d$ {with the same reasoning as for the bi-sphere}. 
\end{itemize} 
\end{proof}
\subsection{Extension to products of distinct sets}
\label{sec:general}
Given a compact set $\X \subset \R^n$, we denote by $\cM(\X)$ the space of positive Borel measures supported on $\X$. 
In this section, we assume that convergence rates have been obtained for polynomial optimization over two given sets $\X_1$ and $\X_2$, thanks to perturbed kernels and associated operators satisfying \eqref{eq:poneref}, \eqref{eq:ptworef} and \eqref{eq:pthreeref}. 
We show how to derive a convergence rate for polynomial optimization over the Cartesian product $\X = \X_1 \times \X_2$. 

{
Before stating our result, we briefly describe the quantities that will be involved. \\
For each set $\X_i$, one associates an integer $m_i$ that is the number of indices needed when decomposing polynomials into  orthonormal components. 
As explained after Theorem \ref{th:generalrate}, this number {is equal to $1$} when $\X_i$ is either the unit sphere, the unit ball or the standard simplex, and is equal to $n$ when $\X_i$ is the hypercube $[-1,1]^n$. 
Its value differs if $\X_i$ is itself a product set. \\
Next, one associates to each set $\X_i$ a function $d_i : \NN \to \NN$ that takes as {input, an integer $t$, and outputs the relaxation order $d_i(t)$ used for the lower bound $\lb(q, \cQ(\X_i))_{d_i(t)}$. 
In the case of the unit sphere, $d_i$ is the identity, but for the bi-sphere, it is twice the identity.} \\
Eventually, one associates to each set $\X_i$ a function $\eta^{(i)} : \NN \to \NN \to \R_{\geq 0}$ converging to 0 at infinity. 
This function is used to quantify the convergence rate. Given a relaxation order $t$ as input, the output is proportional to $1/t^2$  for the unit sphere, the unit ball, the standard simplex and the hypercube, tough the constant of proportionality depends on the set. 
}
\begin{theorem}
\label{th:generalrate}
Let $\X_1, \X_2 \subseteq \R^n$ be two compact semialgebraic sets, and let $\X = \X_1 \times \X_2$ be their Cartesian product. 
For each $i=1,2$, assume that there exist 
\begin{itemize}
\item $m_i\in \NN$, $t_i \in \R_{\geq 0}$,
\item a function $d_i : \NN \to \NN$,
\item a probability measure $\mu_i \in \cM(\X_i)$, 
\item a function $\eta^{(i)} : \NN \to \R_{\geq 0}$, converging to $0$ at infinity, 
\item a decomposition $\cP(\X_i) = \displaystyle\bigoplus_{\kappa^{(i)} \in \NN^{m_i}} H_{\kappa^{(i)}}$, where each $H_{\kappa^{(i)}}$ {is spanned by} a subspace $B_{\kappa^{(i)}}$ of orthonormal polynomials with respect to $(\X_i,\mu_i)$, 
\item a vector $\lambda^{(i)}=\left(\lambda_{\kappa^{(i)}}^{(i)}\right)$  indexed by  $\NN_{2t}^{m_i}$, for all $t \geq t_i$, 
\end{itemize}
such that the following hold
\begin{enumerate}
\item the polynomial kernel $\kers_{i,2t}(\cdot, \cdot; \lambda^{(i)}) : \X_i \times \X_i \to \R$ defined for all $t \geq t_i$ by $$ \kers_{i,2t}(\x,\x'; \lambda^{(i)}) := \sum_{|\kappa^{(i)}|\leq 2t} \lambda_{\kappa^{(i)}}^{(i)} \sum_{P \in B_{\kappa^{(i)}}} P(\x) P(\x'), $$
satisfies $\kers_{i,2t}(\cdot,\x'; \lambda^{(i)}) \in \cQ(\X_i)_{d_i(t)}$, at each fixed $\x' \in \X_i$, for all $t\geq t_i$, 
%
\item 
for all $t \geq t_i$ the vector $\lambda^{(i)}=\left(\lambda_{\kappa^{(i)}}^{(i)}\right)$ satisfies
\begin{align}
\label{eq:lambdaone} \lambda_0^{(i)} = 1, \quad 1/2 \leq \lambda_{\kappa^{(i)}}^{(i)} \leq 1, \quad \text{for all } \kappa^{(i)} \in \NN_{2t}^{m_i},\\
\label{eq:lambdatwo}  1 - \lambda_{\kappa^{(i)}}^{(i)} \leq \eta^{(i)}(t), \quad  \text{for all } 1 \leq |\kappa^{(i)}| \leq 2t. 
\end{align}
\end{enumerate}
Define $m:=m_1+m_2$ and $\eta(t):=\max\{\eta^{(1)}(t),\eta^{(2)}(t)\}$. 
Let $q \in \R[\x_1,\x_2]_d$. 
Then for any $t \geq \max \{t_1,t_2\}$, the lower bound $\lb(q, \cQ(\X))_{{d_1(t)+d_2(t)}}$ for the minimization of $q$ over $\X$ satisfies:
\begin{align}
\label{eq:generalrate}
q_{\min} - \lb(q, \cQ(\X))_{d_1(t)+d_2(t)} \leq 8 \eta(t) \cdot \binom{m+d}{d} \cdot \gamma(\X_1)_d \cdot \gamma(\X_2)_d \cdot (q_{\max} - q_{\min}),
\end{align}
%
%
\end{theorem}
\begin{proof}
Up to translation and scaling, we assume that $q_{\min}=0$ and $q_{\max}=1$. 
Let us fix $t \geq \max\{t_1,t_2\}$. 
We outline how this rate is obtained by adapting the proof of Theorem \ref{th:bisphererate} given in Section \ref{sec:proofbisphere}. 
\begin{itemize}
\item We equip $\X$ with the product measure $\mu_{\X}:=\mu_1 \otimes \mu_2$;
%
\item For each $\kappa=(\kappa^{(1)}, \kappa^{(2)}) \in \NN^m$, one defines $H_{\kappa}[\x_1,\x_2] :=  H_{\kappa^{(1)}}[\x_1] \otimes H_{\kappa^{(2)}}[\x_2] $.  
Doing so, every polynomial $q \in \R[\x_1, \x_2]_d$ has a decomposition $q = \sum_{|\kappa| \leq d} q_{\kappa}$, with $q_{\kappa} \in H_{\kappa}[\x_1, \x_2]$; 
\item For each $\kappa=(\kappa^{(1)}, \kappa^{(2)}) \in \NN^m$, one defines $\lambda_{\kappa} := \lambda_{\kappa^{(1)}} \lambda_{\kappa^{(2)}} $. 
We build a linear invertible operator $\Kprod_{2t} : \R[\x_1,\x_2] \to \R[\x_1, \x_2]$ as follows:
\begin{align*}
\Kprod_{2t} q(\x_1,\x_2) & := \int_{\X} \kerprod_{2t}((\x_1,\x_2),(\x_1',\x_2'); \lambda) q(\x_1',\x_2') \di  \mu_{\X}(\x_1',\x_2'), \\
\text{where }\kerprod_{2t}((\x_1,\x_2),(\x_1',\x_2'); \lambda) & := \kers_{1,2t}(\x_1,\x_1'; \lambda^{(1)}) \kers_{2,2t}(\x_2,\x_2'; \lambda^{(2)}).
\end{align*} 
Then, at each fixed $(\x_1',\x_2') \in \X$, $\kers_{2t}(\cdot,(\x_1',\x_2'); \lambda) \in \cQ(\X)_{d_1(t)+d_2(t)}$ and  $\Kprod_{2t}$ satisfies \eqref{eq:ptworef} {(with $\cQ(\X)_{d_1(t)+d_2(t)}$ instead of $\mathcal{Q}(\X)_{4t}$)}; 
\item Since $\lambda_0=1$, $\Kprod_{2t}$ satisfies \eqref{eq:poneref};
\item One has 
\begin{align}
\label{eq:pthreegeneral}
\Kprod_{2t}^{-1} q - q = \sum_{|\kappa|\leq d}  \left( \frac{1}{\lambda_{\kappa}} -1 \right) q_{\kappa}.
\end{align}
As for \eqref{eq:bernoulli}, we rely on Bernoulli's identity to obtain the inequality $1 - \lambda_{\kappa} \leq 2 \eta(t)$. 
Since  $\lambda_k \geq 1/2$, one has $|1 - \lambda_{\kappa}^{-1}| \leq 4 (1 - \lambda_{\kappa})$. 
Eventually, one has $\|q_{\kappa}\|_{\X}  \leq \gamma(\X_1)_d \cdot \gamma(\X_2)_d$, so $\Kprod_{2t}$ satisfies \eqref{eq:pthreeref}. 
\end{itemize} 
\end{proof}
We now give examples of sets  where the assumptions of Theorem \ref{th:generalrate} are fulfilled. 
\paragraph{The unit sphere \cite{fang2021sum}, the unit ball \cite{slot2022sum} or the standard simplex\cite{slot2022sum}} 
~\\
If $\X_i$ is either  $S^{n-1} := \{\x \in \R^n : \|\x\|^2 = 1\}$, or $\bB^n := \{\x \in \R^n : \|\x\|^2 \leq 1\}$ or $\Delta^n : \{\x \in \R^n : x_1 \geq 0, \dots, x_n \geq 0, \ 1 - \sum_{i=1}^n x_i \geq 0\}$ then one can select
\begin{itemize}
\item $m_i = 1$, $t_i=2 n d \sqrt{d}$;
\item $d_i : = t \mapsto 2 t$;
\item $\mu_i$ to be either the uniform Haar measure for $S^{n-1}$, 
or the measure with density  $c_n (1 - \|\x\|^2)^{-\frac{1}{2}}$ (with $c_n> 0 $ being a normalization constant) for $\bB^n$, 
or the measure  with density $c_n x_1^{-1/2} \cdots x_n^{-1/2} (1 - \sum_{i=1}^n x_i)^{-\frac{1}{2}}$ (with $c_n> 0 $ being a normalization constant) for $\Delta^n$;
\item $\eta^{(i)} : t \mapsto \frac{n^2 d^3}{t^2}$;
\item $\cP(S^{n-1}) = \displaystyle\bigoplus_{\kappa^{(i)}=0}^{\infty} H_{\kappa^{(i)}}$, where $H_{\kappa^{(i)}}$ is { spanned by the} set of orthonormal polynomials with respect to $(S^{n-1},\mu_i)$ of exact degree $\kappa^{(i)}$;
\item the vector $\lambda^{(i)}=\left(\lambda_{\kappa^{(i)}}^{(i)}\right)$, indexed by $\NN_{2t}$, for all $t \geq t_i$, as in Lemma \ref{lemma:boundlambda}. 
\end{itemize}
\paragraph{The hypercube \cite{laurent2023effective}} 
~\\
If $\X_i = [-1, 1]^n$, then one can select
\begin{itemize}
\item $m_i = n$, $t_i= \pi d \sqrt{2n}$;
\item $d_i : t \mapsto n (t+1)$;
\item $\mu_i$ to be the normalized Chebyshev measure with  density  $\frac{1}{\pi \sqrt{1-x_1^2}} \cdots \frac{1}{\pi \sqrt{1-x_n^2}}$;
\item $\eta^{(i)} : t \mapsto \frac{n \pi^2 d^2}{t^2}$, 
\item $\cP([-1,1]^n) = \displaystyle\bigoplus_{\kappa^{(i)} \in \NN^n} H_{\kappa^{(i)}}$, where $H_{\kappa^{(i)}} = \R T_{\kappa^{(i)}}$ and $T_{\kappa^{(i)}}$ is the multivariate Chebyshev polynomial defined as the product of univariate Chebyshev polynomials:
$$
T_{\kappa^{(i)}}(\x) := \prod_{j=1}^n T_{\kappa_j^{(i)}}(x_j), \quad T_k(\cos \phi) := \cos (k \phi) \quad (\phi \in \R, k \in \NN)
;$$
\item the vector $\lambda^{(i)}=\left(\lambda_{\kappa^{(i)}}^{(i)}\right)$, indexed by $\NN^n_{2t}$, for all $t \geq t_i$, as in \cite[Lemma 12]{laurent2023effective}. 
\end{itemize}
{
From the above, we obtain a direct corollary of Theorem \ref{th:generalrate}. 
\begin{corollary}
\label{cor:generalrate}
Let $m\in \NN$, $q \in \R[\x_1,\dots,\x_m]$ and $\X = \X_1 \times \cdots \times \X_m$ be an arbitrary Cartesian product of $m$ sets, assuming that each set $\X_i \subset \R^n$ is either the unit sphere, the unit ball, the standard simplex or the hypercube.  
Then the hierarchy of lower bounds $\lb(q, \cQ(\X))_t$ based on the Schm\"udgen-type certificates converges to the global minimum of $q$ at a rate in $O(1/t^2)$, where $t$ is the relaxation order. 
\end{corollary}
}
\section{Application to quantum Wasserstein distance}
\label{sec:rategmp}
In this section we analyze the convergence rate of the hierarchy derived in \cite{CM25} to approximate as closely as desired the so-called \textit{order 2 quantum Wasserstein distance}. 
Roughly speaking, the classical Wasserstein distance quantifies the cost of transporting one probability measure to another. 
In order to extend this notion to quantum states, several notions of quantum Wasserstein distances have been proposed; see, e.g., \cite{nielsen06} and \cite{de2021quantum}. 
In \cite{CM25}, the authors specifically focus on the order 2 quantum Wasserstein distance that has been recently proposed in \cite{beatty2024}. 
\subsection{Problem statement and main result}
\label{sec:qwstatement}
\paragraph{Additional complex notation.} 
Let $\ib = \sqrt{-1} \in \C$. 
Given $x \in \C$, one can write $x = a + \ib \, b$, where $a = \re(x)$ is the real part of $x$ and $b=\im(x)$ is the imaginary part of $x$. 
The complex conjugate of $x$ is $\overline{x} = a - \ib \, b$. 
We extend this to vectors and matrices by applying the  corresponding operations entrywise. 
The modulus of $x$ is denoted by $|x|:=\sqrt{a^2+b^2}$. 
The set of polynomials in variables $\x=(x_1,\dots,x_n)$ and $\overline{\x}=(\overline{x}_1,\dots,\overline{x}_n)$ with complex coefficients is denoted by $\C[\x,\overline{\x}]$. 
A polynomial $p \in \C[\x,\overline{\x}]$ satisfying $p=\overline{p}$ is called \textit{Hermitian}. 
Given a vector $\qu \in \C^n$,  $\qu^*$ is the row vector containing the conjugate entries of $\qu$. 
We will use the same notation as for the real case for the Euclidean norm $\|\qu\| = \sqrt{\qu^* \qu}$. 
The $\ell_1$-norm $\|\qu\|_1$ is the sum of moduli of the entries of $\qu$. 
A linear functional $L : \C[\x,\overline{\x}] \to \C$ is called Hermitian if $\overline{L(p)} = L(\overline{p})$ for all $p \in \C[\x,\overline{\x}]$. 
Given a compact set $\X \subset \C^n$, we denote by $\cM(\X)$ the space of complex positive Borel measures supported on $\X$. 

Given $n \in \NN$, let $\rho, \nu \in \C^{n \times n}$ be two normalized quantum states, i.e., Hermitian matrices with nonnegative eigenvalues and unit trace. 
A \textit{quantum transport plan} between $\rho$ and $\nu$ is a finite set of triples $\{(\omega_{\ell}, \qu_{\ell}, \qv_{\ell})\}_{\ell}$ such that 
\begin{align}
\sum_{\ell} \omega_{\ell} \, \qu_{\ell} \qu_{\ell}^* = \rho \quad \text{ and} \quad\sum_{\ell} \omega_{\ell} \,  \qv_{\ell} \qv_{\ell}^* = \nu, 
\end{align}
where $\omega_{\ell} > 0$, $\sum_{\ell} \omega_{\ell} =1$, $\qu_{\ell}, \qv_{\ell} \in \C^n$,  $\|\qu_{\ell}\|=\| \qv_{\ell}\| = 1$. 
The set of all such quantum transport plans between two states $\rho$ and $\nu$ is denoted by $\quantum(\rho,\nu)$. 
Given a plan $e=\{(\omega_{\ell}, \qu_{\ell}, \qv_{\ell})\}_{\ell}$, the associated (order $2$) \textit{quantum transport cost} is defined as 
\begin{align}
T_2(e) := \sum_{\ell} \omega_{\ell} \Tr \left[ (\qu_{\ell} \qu_{\ell}^* - \qv_{\ell} \qv_{\ell}^*) \overline{(\qu_{\ell} \qu_{\ell}^* - \qv_{\ell} \qv_{\ell}^*)} \right],
\end{align}
and the order $2$ quantum Wasserstein distance is defined as 
\begin{align}
\label{eq:qwtwo}
W_2(\rho,\nu) := \Big(\inf_{\qt \in \quantum(\rho,\nu)} T_2(\qt)\Big)^{1/2}.
\end{align}
In \cite[Section 5.2]{gribling2022bounding}, the authors propose a moment reformulation of the DPS hierarchy \cite{doherty2002distinguishing}, initially designed to distinguish separable and entangled states. 
Inspired by this reformulation, it is proved in \cite[Theorem 8]{CM25} that $W_2^2(\rho,\nu)$ is the optimal value of a particular instance of the \textit{generalized moment problem} (GMP), i.e., an infinite-dimensional linear problem over positive Borel measures:
\begin{equation}
\label{eq:measC}
\begin{aligned}
W^2_2(\rho,\sigma) = \inf_{\mu \in \cM(\X)} \quad  & \int_{\X} f \di  \mu   \\	
\st 
\quad & \int_{\X} \x \x^* \di  \mu = \rho, \\
\quad & \int_{\X} \y \y^* \di  \mu = \nu,
\end{aligned}
\end{equation}
where $f(\x,\overline{\x},\y,\overline{\y})= \Tr[(\x \x^* - \y \y^*) \overline{(\x \x^* - \y \y^*)}]$ and $\X = \{(\x,\y) \in \C^{2n}: \|\x\|=\|\y\|=1 \}$ is the Cartesian product of two complex unit spheres. 
The dual of problem  \eqref{eq:measC} is given by
\begin{equation}
\label{eq:dualC}
\begin{aligned}
\sup_{\Lambda,\Gamma \in \C^{n \times n}} \quad  & \Tr (\rho \Lambda + \nu \Gamma) \\	
\text{s.t.}
\quad & f - \x^* \Lambda \x -  \y^* \Gamma \y \geq 0 \text{ on } \X, \\
\quad & \Lambda^*=\Lambda, \quad  \Gamma^*=\Gamma. 
\end{aligned}
\end{equation}
The {Putinar-type (or equivalently Schm\"udgen-type)} moment hierarchy for \eqref{eq:measC} is given by
\begin{equation}
\label{eq:momC}
\begin{aligned}
W^2_2(\rho,\sigma)_t := \inf_{L} \quad  & L(f)   \\	
\text{s.t.}
\quad & L: \C[\x,\overline{\x},\y,\overline{\y}]_{2t} \to \C  \text{ Hermitian linear}, \\
\quad & L(\x \x^*) = \rho, \\
\quad & L(\y \y^*) = \nu, \\
\quad & L \geq 0 \text{ on } {\mathcal{Q}(\X)}_{2t}.
\end{aligned}
\end{equation}
It is proved in \cite[Theorem 10]{CM25} that the corresponding sequence of lower bounds given by this hierarchy  converges to the order $2$ quantum Wasserstein distance, i.e., $W^2_2(\rho,\sigma)_t \to W^2_2(\rho,\sigma)$ when $t$ goes to infinity. 
In the recent contribution  \cite{gamertsfelder2025effective}, the authors analyze convergence rates of hierarchies of lower bounds for the GMP. 
We recall this result in the case where the GMP instance involves a single measure and equality constraints only. 

\begin{lemma}[See {Theorem 2.4 and Corollary 2.5 from \cite{gamertsfelder2025effective}}]
\label{lemma:gmp} 
Given $n, d, N \in \NN$, $\tau \in \R^N$, $f, h_1,\dots,h_N \in \R[\x]_d$, a compact semialgebraic set $\X \subset \R^n$, 
let us consider the GMP instance
\begin{equation}
\label{eq:primalgmp}
\begin{aligned}
\primal := \inf_{\mu \in \cM(\X)} \quad  & \int_{\X} f \di  \mu   \\	
\st 
\quad & \int_{\X} h_i \di  \mu = \tau_i, \quad i=1,\dots,N, 
\end{aligned}
\end{equation}
assumed to have a nonempty feasible set.  
The associated moment hierarchy of {Putinar}-type bounds is given by
\begin{equation}
\label{eq:momgmp}
\begin{aligned}
\primal_t := \inf_{L} \quad  & L(f)   \\	
\st 
\quad & L: \R[\x]_{2t} \to \R  \text{ linear}, \\
\quad & L(h_i) = \tau_i, \quad i=1,\dots,N, \\
\quad & L \geq 0 \text{ on } {\mathcal{Q}}(\X)_{2t}.
\end{aligned}
\end{equation}
Assume that ${\mathcal{Q}(\X)}$ is Archimedean, and that a convergence rate is given for the hierarchy of {Putinar}-type bounds for polynomial optimization on $\X$, i.e., there exists nonnegative constants (depending on $n$, $d$, $\X$) $C_{\X}$ and $t_0$ such that for all $q \in \R[\x]_d$ and $t \geq t_0$, one has 
$q_{\min} - \lb(q, {\mathcal{Q}(\X)})_{t} \leq C_{\X} \eta(t) \cdot (q_{\max} - q_{\min})$, with $\eta : \NN \to \R_{\geq 0}$ being a function converging to $0$ at infinity. 
Assume that there exists $w \in \R^n$ such that $\sum_{i=1}^N w_i h_i > 0$ on $\X$, {and that the dual of \eqref{eq:primalgmp} attains its supremum}. 

Then for all $t \geq t_0$, one has 
\[
0 \leq \primal - \primal_t \leq \frac{1 + \sum_{i=1}^N \tau_i w_i}{(\sum_{i=1}^N w_i h_i)_{\min}} C_{\X}  \eta(t) \left({2} \|f\|_{\X} + \|w^{\opt}\|_1 \|h\|_{\X}\right),
\]
where $\|h\|_{\X} := \max_{1\leq i \leq N} \|h_i\|_{\X}$, and $w^{\opt}$ is a vector achieving the supremum in the dual of \eqref{eq:primalgmp}. 

\end{lemma}
Our final result, given in Theorem \ref{th:qwratereal} below, provides the convergence rate of the moment hierarchy \eqref{eq:momC} for the order $2$ quantum Wasserstein distance, assuming that there is a bounded dual optimal solution of \eqref{eq:dualC}. 
This result is obtained by combining Lemma \ref{lemma:gmp} with our analysis derived in Theorem \ref{th:bisphererate}. 
\if{ 
Let us denote by $\rho_{\min}$ and $\rho_{\max}$ the minimal and maximal respective eigenvalues of $\rho$, respectively. We define similarly $\nu_{\min}$ and $\nu_{\max}$. 
To bound the dual solution of \eqref{eq:dualC}, we will assume the following. 
\begin{assumption}
\label{hyp:bound}
Both $\rho$ and $\nu$ have full rank and either $\nu_{\min} > \rho_{\max}$ or $\rho_{\min} > \nu_{\max}$. 
\end{assumption}
}\fi
\begin{theorem}
\label{th:qwratereal}
Let assume that the dual \eqref{eq:dualC} is attained at a bounded tuple. 
For any $t \geq 32 n$, one has
\begin{align}
\label{eq:qwrate}
0 \leq W^2_2(\rho,\nu) - W^2_2(\rho,\nu)_{{t}} \leq \frac{\kappa(n,\rho,\nu)}{t^2},
\end{align}
where $\kappa(n,\rho,\nu)$ is a constant depending only on $n$, $\rho$ and $\nu$, and being proportional to the constant $C_{\X}(2n,4)$ from Theorem \ref{th:bisphererate}. See relation \eqref{eq:kappa} for details. 
\end{theorem}
The proof of Theorem \ref{th:qwratereal} is postponed in Section \ref{sec:proofqw}. 
{
As shown in \cite[Theorem 9]{CM25}, when both $\rho$ and $\nu$ are  positive definite quantum states, then the dual \eqref{eq:dualC} is attained at a bounded tuple.  In this case, the constant $\kappa$ depends on $n$ and the reciprocal of the states' minimal eigenvalues. 
}
\if{ 
Note that the hypotheses of Assumption \ref{hyp:bound} on $\rho$ and $\nu$ are rather restrictive. 
One can instead simply assume that there is an optimal solution $w^{\opt} = (\Lambda,\Gamma)$ for the dual \eqref{eq:dualC} with bounded entries, in which case the constant $\kappa(n)$ is given as in \eqref{eq:kappafirst}. 
}\fi
\subsection{Changing variables from complex to real}
\label{sec:complextoreal}
In order to apply Lemma \ref{lemma:gmp}, we provide an equivalent formulation of the moment hierarchy \eqref{eq:momC} involving linear functionals acting on real polynomials. 
This reformulation is obtained thanks to \cite[Appendix A]{gribling2022bounding},  where the authors describe how to transform  polynomials in $\C[\x,\overline{\x},\y,\overline{\y}]$ into real polynomials in $\R[\ab,\bb,\cb,\db]$ via the change of variables $\x = \ab + \ib \, \bb$, $\y = \cb + \ib \, \db$. 
In this way, any polynomial $p \in \C[\x,\overline{\x},\y,\overline{\y}]$ corresponds to a unique pair of real polynomials 
\begin{align}
p_{\re} (\ab,\bb,\cb,\db) = \re(p(\ab+\ib \, \bb, \ab-\ib \, \bb, \cb+\ib \, \db, \cb-\ib \, \db )), \\
p_{\im} (\ab,\bb,\cb,\db) = \im(p(\ab+\ib \bb, \ab-\ib \, \bb, \cb+\ib \, \db, \cb-\ib \, \db )), 
\end{align}
satisfying $p(\x,\overline{\x},\y,\overline{\y})= p_{\re} (\ab,\bb,\cb,\db) + \ib \, p_{\im} (\ab,\bb,\cb,\db)$. 
When $p$ is Hermitian, one has $p(\x,\overline{\x},\y,\overline{\y}) = p_{\re} (\ab,\bb,\cb,\db)$, so in particular one has
$$f(\x,\overline{\x},\y,\overline{\y})= \Tr[(\x \x^* - \y \y^*) \overline{(\x \x^* - \y \y^*)}] = f_{\re} (\ab,\bb,\cb,\db) .$$
Given a set of Hermitian polynomials, we define its real analog by applying the map $\re(\cdot)$ entrywise to this set. 
In particular, the real analog of the complex bi-sphere $\X$ is the set $\X_{\re} = \{ (\ab,\bb,\cb,\db) \in \R^{4n} : 1 - \|\ab\|^2 - \|\bb\|^2 = 1 - \|\cb\|^2 - \|\db\|^2 = 0  \}$. 
The associated truncated real quadratic module is 
\begin{align*}
{\mathcal{Q}}^{\R}(\X_{\re})_{2t} := 
\left\{ 
\sigma + q_1 (1 - \|\ab\|^2 - \|\bb\|^2) + q_2 (1 - \|\cb\|^2 - \|\db\|^2) :  \right. & \sigma \in \Sigma[\ab,\bb,\cb,\db]_{2t}, \\
& \left. q_1,q_2 \in \R[\ab,\bb,\cb,\db]_{2t-2} 
\right\} .
\end{align*}
Similarly, the maps $\re(\cdot)$ and $\im(\cdot)$ act entrywise for vectors and matrices with polynomial entries in  $\C[\x,\overline{\x},\y,\overline{\y}]$, e.g., $\re(\x \x^*) = \ab \ab^T + \bb \bb^T$ and $\im(\x \x^*) = \bb \ab^T - \ab \bb^T$. 
Eventually, for a linear functional $L : \C[\x,\overline{\x},\y,\overline{\y}] \to \C$, one has $L(p) = \re(L(p))+ \ib \, \im(L(p)) $. If in addition $L$ is Hermitian, that is, satisfies $\overline{L(p)}=L(\overline{p})$, then one can associate to $L$ a real linear functional $L^{\R} : \R[\ab,\bb,\cb,\db] \to \R$, defined by
\[
L^{\R} (q) = L\left(  q \left(  \frac{\x+\overline{\x}}{2},   \frac{\x-\overline{\x}}{2 \ib},   \frac{\y+\overline{\y}}{2},   \frac{\y-\overline{\y}}{2 \ib}   \right) \right) .
\]
Then for any $p \in \C[\x,\overline{\x},\y,\overline{\y}]$, one has $L(p) = L^{\R} (p_{\re}) + \ib \,  L^{\R} (p_{\im})$. 
For any $t \geq 2$, we can now provide the following relaxation over real linear functionals:
\begin{equation}
\label{eq:momR}
\begin{aligned}
\inf_{L^{\R}} \quad  & L^{\R}(f^{\re})   \\	
\text{s.t.}
\quad & L^{\R}: \R[\ab,\bb,\cb,\db]_{2t} \to \R  \text{ linear}, \\
\quad & L^{\R}(\ab \ab^T + \bb \bb^T) = \re(\rho), \\
\quad & L^{\R}(\bb \ab^T - \ab \bb^T) = \im(\rho), \\
\quad & L^{\R}(\cb \cb^T + \db \db^T) = \re(\nu), \\
\quad & L^{\R}(\db \cb^T - \cb \db^T) = \im(\nu), \\
\quad & L^{\R} \geq 0 \text{ on } \cQ^{\R}(\X_{\re})_{2t}.
\end{aligned}
\end{equation}
As a consequence of the above defined transformations, the relaxation \eqref{eq:momR} is equivalent to the moment relaxation \eqref{eq:momC} over Hermitian linear functionals. 
In particular the equivalence  
\[
L \geq 0 \text{ on } {\mathcal{Q}}(\X)_{2t} \iff L^{\R} \geq 0 \text{ on } {\mathcal{Q}}^{\R}(\X_{\re})_{2t} 
\]
is proved in \cite[Lemma 33]{gribling2022bounding}. 
As in the complex setting, one can show that the hierarchy of relaxations converges to $W^2_2(\rho,\nu)$ when $t$ goes to infinity, that is
\begin{equation}
\label{eq:infmomR}
\begin{aligned}
W^2_2(\rho,\nu) = \inf_{L^{\R}} \quad  & L^{\R}(f_{\re})   \\	
\text{s.t.}
\quad & L^{\R}: \R[\ab,\bb,\cb,\db] \to \R  \text{ linear}, \\
\quad & L^{\R}(\ab \ab^T + \bb \bb^T) = \re(\rho), \\
\quad & L^{\R}(\bb \ab^T - \ab \bb^T) = \im(\rho), \\
\quad & L^{\R}(\cb \cb^T + \db \db^T) = \re(\nu), \\
\quad & L^{\R}(\db \cb^T - \cb \db^T) = \im(\nu), \\
\quad & L^{\R} \geq 0 \text{ on } {\mathcal{Q}}^{\R}(\X_{\re}).
\end{aligned}
\end{equation}
The latter formulation can be written equivalently as an infinite-dimensional linear problem over real positive Borel measures supported on $\X_{\re}$:
\begin{equation}
\label{eq:measR}
\begin{aligned}
W^2_2(\rho,\nu) = \inf_{\mu^{\R}\in \cM(\X_{\re})} \quad  & \int f_{\re} \di \mu^{\R}   \\	
\text{s.t.}
\quad & \int (\ab \ab^T + \bb \bb^T) \di \mu^{\R} = \re(\rho), \\
\quad & \int (\bb \ab^T - \ab \bb^T) \di \mu^{\R} = \im(\rho), \\
\quad & \int (\cb \cb^T + \db \db^T) \di \mu^{\R} = \re(\nu), \\
\quad & \int (\db \cb^T - \cb \db^T) \di \mu^{\R} = \im(\nu). \\
\end{aligned}
\end{equation}
Note that every feasible solution $\mu_{\R}$ of \eqref{eq:measR} must be a probability measure. 
Indeed, combining the first equality constraint together with the fact that $\mu_{\R}$ is supported on $\X_{\re}$ yields
\[
\int \Tr(\ab \ab^T + \bb \bb^T) \di \mu^{\R} = \int (\|\ab\|^2+\|\bb\|^2) \di \mu^{\R} = \int \di \mu^{\R} = \Tr(\re(\rho))=1.\] 
The dual of \eqref{eq:measR} is given by
\begin{equation}
\label{eq:dualR}
\begin{aligned}
\sup_{\Lambda_i,\Gamma_i \in \R^{n \times n}} \quad  & \Tr (\re(\rho) \Lambda_1 + \re(\nu) \Gamma_1 + \im(\rho) \Lambda_2 + \im(\nu) \Gamma_2) \\	
\text{s.t.}
\quad & f_{\re} \geq \Tr ((\ab \ab^T + \bb \bb^T)\Lambda_1 + (\cb \cb^T + \db \db^T)\Gamma_1 
 + (\bb \ab^T - \ab \bb^T)\Lambda_2 \\
\quad & \quad \quad \quad + (\db \cb^T - \cb \db^T)\Gamma_2) \text{ on } \X_{\re}, \\
\quad & \Lambda_1^T=\Lambda_1, \quad \Lambda_2^T=-\Lambda_2, \quad \Gamma_1^T=\Gamma_1, \quad \Gamma_2^T=-\Gamma_2. 
\end{aligned}
\end{equation}
\subsection{Proof of Theorem \ref{th:qwratereal}}
\label{sec:proofqw}
We are now in the appropriate setup to apply Lemma \ref{lemma:gmp}. 
\begin{proof}
The three required assumptions are as follows:
\begin{enumerate}
\item the instance of the generalized moment problem has a non-empty feasible set;
\item there exists $(\Lambda_i, \Gamma_i)$ such that 
$\Tr ((\ab \ab^T + \bb \bb^T)\Lambda_1 + (\cb \cb^T + \db \db^T)\Gamma_1 + (\bb \ab^T - \ab \bb^T)\Lambda_2 + (\db \cb^T - \cb \db^T)\Gamma_2)$ 
is positive on  $\X_{\re}$;
\item the quadratic module ${\mathcal{Q}}^{\R}(\X_{\re})$ is Archimedean. 
\end{enumerate} 
The first assumption is true because the original problem has a non-empty feasible set $\quantum(\rho,\nu)$: for any spectral decompositions $\rho = \sum_{\ell} \omega_{\ell} \qu_{\ell} \qu_{\ell}^*$ and $\nu = \sum_{j} \omega'_{j} \qv_j \qv_j^*$, one has $\{(\omega_{\ell} \omega'_{j}, \qu_{\ell}, \qv_{j})_{\ell,j}\} \in \quantum(\rho,\nu)$. \\
The second assumption holds with $\Lambda_1=\Gamma_1=I_n$ and  $\Lambda_2=\Gamma_2=0$, since the polynomial
$\Tr (\ab \ab^T + \bb \bb^T + \cb \cb^T + \db \db^T) = {\|\ab\|^2 + \|\bb\|^2 + \|\cb\|^2 + \|\db\|^2}$ is equal to $2 > 0$ on $\X_{\re}$. \\
The third assumption is obviously true since the polynomial $2 - {\|\ab\|^2 - \|\bb\|^2 - \|\cb\|^2 - \|\db\|^2} $ belongs to ${\mathcal{Q}}^{\R}(\X_{\re})$. 

{The result from Lemma \ref{lemma:gmp} can then be applied together with Theorem \ref{th:bisphererate} (see Remark \ref{rk:bisphererate}) to obtain that for all $t \geq 4 n d \sqrt{d} = 32 n$, one has}
\begin{align}
\label{eq:kappafirst}
0 \leq W^2_2(\rho,\nu) - W^2_2(\rho,\nu)_{{t}} \leq \frac{\kappa(n)}{t^2}, \nonumber \\
\text{with } \kappa(n) := {4}  \frac{3}{2} C_{\X_{\re}(2n,4)} \left({2} f_{\re,\max} + \|w^{\opt}\|_1 h_{\max}\right),
\end{align}
where
\begin{itemize}
\item the numerator of the constant $\frac{3}{2}$ has been obtained by adding 1 to  the evaluation of the dual objective function at $\Gamma_1=\Lambda_1=I_n$   and $\Lambda_2=\Gamma_2=0$, that is the tuple we used above to satisfy  the last assumption from Lemma \ref{lemma:gmp}. 
The denominator corresponds to the minimum of $\Tr ((\ab \ab^T + \bb \bb^T) + (\cb \cb^T + \db \db^T) = \|\ab\|^2 + \|\bb\|^2 + \|\cb\|^2 + \|\db\|^2$  on $\X_{\re}$;
\item the constant $f_{\re,\max}$ is the maximum of $f_{\re}$ on $\X_{\re}$;
\item the constant $\|w^{\opt}\|_1$ is the sum of absolute values of all entries of a tuple $(\Lambda_1,\Lambda_2,\Gamma_1,\Gamma_2)$ achieving the supremum in the dual \eqref{eq:dualR}; 
\item the constant $h_{\max}$ is the maximum value over $\X_{\re}$ {achieved by the polynomials involved in the equality constraints of the primal \eqref{eq:measR}, i.e., by terms of the form $a_i a_j + b_i b_j$, $a_i b_j - b_j a_i$, $c_i c_j + d_i d_j$, or $c_i d_j - d_j c_i$}. 
\end{itemize}
We end the proof by providing estimates for the 3 quantities $f_{\re,\max}$, $\|w^{\opt}\|_1$ and $h_{\max}$. 

1. The first quantity $f_{\re,\max}$ is equal to the maximum $f_{\max}$ of $f$ on $\X$. 
One has 
\begin{align*}
f(\x,\overline{\x},\y,\overline{\y})= \Tr[(\x \x^* - \y \y^*) \overline{(\x \x^* - \y \y^*)}] & = \left|\sum_{i=1}^n x_i^2\right|^2 + \left|\sum_{i=1}^n y_i^2\right|^2 - 2 \left|\sum_{i=1}^n x_i y_i\right|^2 \\
& \leq \left(\sum_{i=1}^n |x_i|^2\right)^2 + \left(\sum_{i=1}^n |y_i|^2\right)^2.
\end{align*}
This shows that $f_{\max} \leq 2$. In addition, $f(\x,\overline{\x},\y,\overline{\y}) = 2$  for $\x=(1,0,\dots,0)$ and $\y=(0,1,0,\dots,0)$. Therefore, one has $f_{\max}=f_{\re,\max}=2$.

2. Since the complex dual \eqref{eq:dualC} is attained at a bounded tuple $(\Lambda,\Gamma)$, the real dual \eqref{eq:dualR} is also attained at a bounded tuple $(\Lambda_1,\Lambda_2,\Gamma_1,\Gamma_2)$, with $\Lambda=\Lambda_1 + \ib \, \Lambda_2$ and $\Gamma=\Gamma_1 + \ib \, \Gamma_2$. \\ 
For all $z = a + \ib \, b$ with $a,b \in \R$, one has $|a| + |b| \leq \sqrt{2} \sqrt{a^2 + b^2} = \sqrt{2} |z|$, thus $\|\Lambda_1\|_1 + \|\Lambda_2\|_1 \leq \sqrt{2} \| \Lambda \|_1$, and $\|\Gamma_1\|_1 + \|\Gamma_2\|_1 \leq \sqrt{2} \| \Gamma \|_1$. 

\if{ 
2. Under Assumption \ref{hyp:bound}, let us handle specifically the case where $\nu_{\min} > \rho_{\max}$,  the other case can be treated analogously. 
Let us show that in this case the second quantity $\|w^{\opt}\|_1$ is less than $2 \sqrt{2} n \left(1 -  \frac{\rho_{\max}^2}{\nu_{\min}^2}\right)$. 
Note that every optimal tuple $(\Lambda_1,\Lambda_2,\Gamma_1,\Gamma_2)$ for the real dual \eqref{eq:dualR} corresponds to an optimal tuple $(\Lambda,\Gamma)$ for the complex dual \eqref{eq:dualC}, with $\Lambda=\Lambda_1 + \ib \, \Lambda_2$ and $\Gamma=\Gamma_1 + \ib \, \Gamma_2$. \\ 
For all $z = a + \ib \, b$ with $a,b \in \R$, one has $|a| + |b| \leq \sqrt{2} \sqrt{a^2 + b^2} = \sqrt{2} |z|$, thus $\|\Lambda_1\|_1 + \|\Lambda_2\|_1 \leq \sqrt{2} \| \Lambda \|_1$, and $\|\Gamma_1\|_1 + \|\Gamma_2\|_1 \leq \sqrt{2} \| \Gamma \|_1$. 
For all $\delta > 0$, the tuple $(\Lambda + \delta I,\Gamma - \delta I)$ is also optimal for \eqref{eq:dualC}  since 
\[\Tr (\rho (\Lambda + \delta I) + \nu (\Gamma - \delta I )) 
= \Tr (\rho \Lambda  + \nu \Gamma) + \delta \Tr (\rho - \nu) ) = \Tr (\rho \Lambda  + \nu \Gamma),\] and 
\[\x^* (\Lambda + \delta) I \x -  \y^* (\Gamma - \delta I) \y = \x^* \Lambda \x -  \y^* \Gamma \y,\] for all $(\x,\y) \in  \X$. 
Thus, we assume without loss of generality that $\Lambda, -\Gamma \succeq 0$. 
Let us note $\mathbf{e}_i$ the vector having only 0 components except a component 1 in the $i$-th place. 
Since the polynomial $f - \x^* \Lambda \x - \y^* \Gamma \y $ is nonnegative on $\X$, one has in particular for $\x = \y = \mathbf{e}_i$:
\[
    \Lambda_{ii} +  \Gamma_{ii} \leq f_{\max} = 2,
\]
implying that $\Tr(\Lambda) \leq 2 - \Tr(\Gamma)$. 
In addition, the objective function $\Tr(\rho \Lambda) + \Tr(\nu \Gamma)$ is always nonnegative, thus 
\[\rho_{\max} \Tr(\Lambda) \geq \Tr(\rho \Lambda) \geq - \Tr(\nu \Gamma) \geq - \nu_{\min} \Tr(\Gamma).\]
After combining the two above inequalities we obtain $\Tr(\Lambda) - \Tr (\Gamma) \leq 2 \left(1 -  \frac{\rho_{\max}^2}{\nu_{\min}^2}\right)$. 

The desired bound follows then from the Cauchy-Schwartz inequality:
\begin{align*}
\|\Lambda_1\|_1 + \|\Lambda_2\|_1 + \|\Gamma_1\|_1 + \|\Gamma_2\|_1
& \leq \sqrt{2} \left(\|\Lambda\|_1 +  \|\Gamma\|_1 \right)\\
& \leq \sqrt{2} \sqrt{n}  \left(\sqrt{\Tr(\Lambda^2)} + \sqrt{\Tr(\Gamma^2)}\right) \\
& \leq \sqrt{2} n( \Lambda_{\max} - \Gamma_{\min}) \leq \sqrt{2} n (\Tr(\Lambda) - \Tr (\Gamma)). 
\end{align*}
where  $\Lambda_{\max}$ and $\Gamma_{\min}$ denotes the largest and smallest eigenvalues of $\Lambda$ and $\Gamma$, respectively.
}\fi

3. The third quantity $h_{\max}$ is equal to 2, since each term of the form $a_i a_j + b_i b_j$, $a_i b_j - b_j a_i$, $c_i c_j + d_i d_j$, or $c_i d_j - d_j c_i$, is upper bounded by 2 on $\X_{\re}$.  \\
We have just proved that 
\begin{align}
\label{eq:kappa}
\kappa(n) \leq  {12 \left(2+ \sqrt{2} (\|\Lambda\|_1 + \|\Gamma\|_1) \right)} C_{\X_{\re}(2n,4)}.
\end{align}
\if{ 
We have just proved that 
\begin{align}
\label{eq:kappa}
\kappa(n) \leq  \frac{3}{2} \left(3+4 \sqrt{2} n \left(1 -  \frac{\rho_{\max}^2}{\nu_{\min}^2}\right) \right) C_{\X_{\re}(2n,4)}.
\end{align}
}\fi
\end{proof}
\section{Concluding remarks}
We have proved a convergence rate in $O(1/t^2)$ for the Schm\"udgen-type (or equivalently Putinar-type)   hierarchies of lower bounds for the minimization of a polynomial over the bi-sphere. 
Thanks to this result, we could provide a similar rate for a hierarchy of semidefinite programs approximating the $2$ order quantum Wasserstein distance. 
We also extended this result to arbitrary sphere products, and products of distinct sets while assuming that a rate is available for each set. 
Our proof technique heavily relies on the polynomial kernel method. 

We would like to emphasize that all obtained convergence rates for hierarchies of lower bounds also hold for the hierarchies of \textit{upper} bounds developed in \cite{lasserre2011new}, by following the same reasoning as the one from \cite[Section 5]{slot2022sum}. 
In addition, our derived rates can be combined with the general result from \cite{gamertsfelder2025effective} to obtain rates for GMP instances with measures supported on set products. 

Recently, the convergence rate of the Putinar-type hierarchy for polynomial optimization over the hypercube has been revisited in \cite{gribling2025revisiting}. 
The idea is also to rely on the polynomial kernel method with a suitably chosen Gaussian density. 
An interesting open question would be to investigate whether this idea could be extended for set products. 

For the case of the unit sphere, we would like also to mention the recent works   \cite{lovitz2023hierarchy,blomenhofer2024moment}, where the authors provide convergence rates in $O(1/t)$ for a specialized hierarchy of spectral bounds, using quantum and real de Finetti theorems, respectively. 
A natural research track would be to extend this analysis to the case of the bi-sphere.

Eventually, it has been proved in \cite{nie2014optimality} that the Putinar-type hierarchy converges in finitely many steps, under genericity assumptions on the objective function and the polynomials describing the feasible set. 
When optimizing a generic polynomial over the unit sphere, it has been shown in \cite{huang2023optimality} that one also has finite convergence. 
Here again, it would be interesting to obtain similar finite convergence results for specific set products such as the bi-sphere. 


\paragraph{Acknowledgments.} We thank Jonas Britz, Saroj Prasad Chhatoi, Monique Laurent and Lucas Slot for fruitful discussions. 
This work was supported by the European Union's HORIZON--MSCA-2023-DN-JD programme under the Horizon Europe (HORIZON) Marie Sklodowska-Curie Actions, grant agreement 101120296 (TENORS). 



\end{document}